\documentclass[10pt]{amsart}
\usepackage{amssymb}
\usepackage{color}
\usepackage{graphicx}
\usepackage{float}
\usepackage{tikz}
\usetikzlibrary{matrix}
\usepackage{epstopdf}

\newcommand{\SE}{{\mathcal{E}}}
\newcommand{\SD}{{\mathcal{D}}}
\newcommand{\SW}{{\mathcal{W}}}

\newcommand{\Gr}{\operatorname{Gr}}
\newcommand{\V}{\operatorname{\mathcal{V}}}

\newcommand{\PGL}{\operatorname{\mathbb{P}GL}}
\newcommand{\Diff}{\operatorname{Diff}}
\newcommand{\SO}{\operatorname{SO}}
\newcommand{\Cover}{\operatorname{\mathcal{C}over}}
\newcommand{\Maps}{\operatorname{\mathcal{M}aps}}

\newcommand{\trivial}{{\operatorname{trivial}}}

\newcommand{\Engel}{\operatorname{\mathfrak{Engel}}}
\newcommand{\FEngel}{\operatorname{\mathfrak{FEngel}}}

\newcommand{\Cartan}{\operatorname{\mathfrak{Cartan}}}
\newcommand{\FCartan}{\operatorname{\mathfrak{FCartan}}}

\newcommand{\Cont}{\operatorname{\mathcal{C}-\mathcal{S}\mathfrak{trs}}}
\newcommand{\Planes}{\operatorname{\mathcal{P}\mathfrak{lanes}}}
\newcommand{\Bundles}{\operatorname{\mathcal{B}\mathfrak{undles}}}

\newcommand{\SG}{{\mathcal{G}}}

\newcommand{\ST}{{\mathcal{T}}}
\newcommand{\s}{{\mathfrak{s}}}
\newcommand{\SU}{{\mathcal{U}}}

\newcommand{\SL}{\operatorname{\mathcal{L}}}

\newcommand{\Imm}{\operatorname{\mathcal{I}}}
\newcommand{\FImm}{\operatorname{\mathfrak{F}\mathcal{I}}}

\newcommand{\R}{{\mathbb{R}}}
\newcommand{\Z}{{\mathbb{Z}}}
\newcommand{\NS}{{\mathbb{S}}}
\newcommand{\D}{{\mathbb{D}}}

\newcommand{\Op}{{\mathcal{O}p}}

\newtheorem{lemma}{Lemma}
\newtheorem{proposition}{Proposition}
\newtheorem{theorem}{Theorem}
\newtheorem*{theorem*}{Theorem}

\newtheorem{definition}{Definition}

\theoremstyle{definition}

\newtheorem{remark}{Remark}

\newtheorem{example}{Example}

\begin{document} 

\title{On the classification of prolongations up to Engel homotopy}

\subjclass[2010]{Primary: 58A30.}
\date{\today}

\keywords{Engel structure, $h$--principle, Cartan prolongation, Lorentzian prolongation}

\author{\'Alvaro del Pino}
\address{Universidad Aut\'onoma de Madrid and Instituto de Ciencias Matem\'aticas -- CSIC.
C. Nicol\'as Cabrera, 13--15, 28049, Madrid, Spain.}
\email{alvaro.delpino@icmat.es}

\begin{abstract}
In \cite{CPPP} it was shown that Engel structures satisfy an existence $h$--principle, and the question of whether a full $h$--principle holds was left open. In this note we address the classification problem, up to Engel deformation, of Cartan and Lorentz prolongations. We show that it reduces to their formal data as soon as the turning number is large enough.

Somewhat separately, we study the homotopy type of the space of Cartan prolongations, describing completely its connected components in the overtwisted case.
\end{abstract}

\maketitle

\section{Introduction. What can you find in this paper?}

A $2$--plane field in a $4$--manifold is called an \textsl{Engel structure} if it is everywhere maximally non--integrable. These structures conform the one exceptional family in Cartan's list of topologically stable distributions: unlike the other structures in the list -- \emph{line fields}, \textsl{contact structures}, and \textsl{even-contact structures} -- Engel structures are a phenomenon particular to a single dimension, dimension $4$. For a long time, except for a few constructions arising from contact and lorentzian geometry, not much was known about their existence and classification.

An Engel structure induces in its ambient manifold a complete flag satisfying some compatibility conditions; this we call a \textsl{formal Engel structure}. Under orientability assumptions, this yields a parallelization. A first breakthrough came with Vogel's thesis \cite{Vo}, in which he was able to show that any parallelizable manifold admits an Engel structure. Then, in \cite{CPPP} it was shown that the natural inclusion $\Engel \to \FEngel$ of the space of Engel structures into the space of formal Engel structures is a surjection in homotopy groups. Whether an $h$--principle relative to the parameter holds remains an open question. 

This note aims to provide some insight into the classification problem by particularising to the case of \textsl{Cartan} and \textsl{Lorentz prolongations}: the classical examples of Engel structures arising on $\mathbb{S}^1$--bundles as projectivisations of contact and lorentzian manifolds, respectively. 

In Section \ref{sec:prolongations} we describe the homotopy type of the space of Cartan prolongations $\Cartan$. Denote by $\Cartan(\xi)$ the prolongations lifting the contact structure $\xi$; denote by $\Cartan([\xi])$ those prolongations lifting contact structures isotopic to $\xi$. We are able to compute the homotopy groups of $\Cartan([\xi])$; this is particularly simple when $\xi$ is overtwisted (see Theorem \ref{thm:ot}). Klukas and Sahamie had already described the connected components of $\Cartan(\xi)$ in \cite{KS}.

Having understood homotopies through Cartan prolongations, which is a more restrictive case of independent interest, in Section \ref{sec:Engel} we turn to homotopies through Engel structures. Prolongations have a well defined invariant called the \textsl{turning number}. Theorem \ref{thm:main} shows that, as soon as the turning number is large enough, the classification question reduces to the classification as formal Engel structures. The main ingredient is the work of Little \cite{Li} and Saldanha \cite{Sal}.

\textbf{Acknowledgements:} The author would like to thank F. Presas for encouraging him to write this note and for his comments on the proof of Proposition \ref{prop:planeEulerClass}. The author is also grateful to the anonymous referee. The author is supported by the Spanish Research Projects SEV--2015--0554, MTM2013--42135, and MTM2015--72876--EXP and a La Caixa--Severo Ochoa grant.

\section{Preliminaries} \label{sec:preliminaries}

Henceforth all manifolds and distributions considered will be smooth. Unless explicitely stated otherwise, manifolds will be closed. To simplify the discussion, both manifolds and distributions will be orientable and often oriented. Our arguments would carry through taking suitable double or quadruple covers in the non--orientable case. The spaces of maps considered are endowed with the $C^\infty$--topology.

Given two distributions $\eta$ and $\nu$ over the manifold $M$, we write:
\[ [\eta, \nu] = \cup_{p \in M} \{ [u,v]_p \in T_pM \text{ $|$ } u \in \Gamma(\eta), v \in \Gamma(\nu) \} \subset TM \]
for their Lie bracket, which is not necessarily a distribution. Note that $\eta \subset [\eta, \eta]$.

\subsection{Definitions}

\begin{definition}
Let $N$ be a $3$--dimensional manifold. A $2$--dimensional distribution $\xi \subset TN$ is said to be a \emph{contact structure} if it is everywhere non--integrable. That is, $[\xi,\xi] = TN$. 
\end{definition} 

\begin{definition}
Let $M$ be a $4$--dimensional manifold. A $3$--dimensional distribution $\SE \subset TM$ is said to be an \emph{even--contact structure} if it is everywhere non--integrable, i.e. if $[\SE,\SE] = TM$. 
\end{definition}

\begin{definition}
Let $M$ a $4$--dimensional manifold. A $2$--dimensional distribution $\SD \subset TM$ is said to be an \emph{Engel structure} if it is everywhere maximally non--integrable, i.e. if $\SE = [\SD, \SD]$ is an even--contact structure. 
\end{definition}

Before we discuss what the state of the art is regarding Engel structures, let us recall some standard results:

\begin{proposition} \label{prop:str}
Let $M$ be a $4$--dimensional manifold. Let $\SE \subset TM$ be an even--contact structure.
\begin{itemize}
\item There is a uniquely defined line field $\SW \subset \SE$ given by the equation $[\SW,\SE] \subset \SE$. $\SW$ is called the \textsl{kernel} of the even contact structure.
\item Given some Engel structure $\SD \subset TM$ satisfying $\SE = [\SD, \SD]$, it holds that $\SW \subset \SD$.
\item Let $N \subset M$ be a (possibly open) $3$--dimensional submanifold of $M$ that is transverse to $\SW$. Then, $\xi = TN \cap \SE$ is a contact structure in $N$. Additionally, given $\SD$ as above, $X = TN \cap \SD \subset \xi$ is a distinguished legendrian line field. 
\item There is a canonical isomorphism given by Lie bracket:
\begin{equation} \label{eq:C1} 
	\det(\SE/\SW) = TM/\SE.
\end{equation}
Additionally, given $\SD$ as above, there is a second isomorphism:
\begin{equation} \label{eq:C2} 
	\det(\SD) = \SE/\SD. 
\end{equation}
\end{itemize}
\end{proposition}

Equation (\ref{eq:C1}) shows that orientability of $TM$ is equivalent to orientability of $\SW$. If $\SE$ arises from some Engel structure $\SD$, $\SE$ is canonically oriented by Equation (\ref{eq:C2}). Hence, choosing orientations for $\SD$ and $M$ yields a parallelisation of $M$ up to homotopy. We are interested precisely in this case.

\subsection{Formal Engel structures}

Proposition \ref{prop:str} motivates the following definition:
\begin{definition}
Let $M$ be a $4$--manifold. A complete flag $\SW \subset \SD \subset \SE \subset TM$ endowed with bundle isomorphisms as in Equations (\ref{eq:C1}) and (\ref{eq:C2}) is said to be a \emph{formal Engel structure}.
\end{definition}

The following was the main result in \cite{CPPP}:
\begin{proposition}
Let $M$ be a smooth $4$--manifold. The inclusion
\[ i: \Engel(M)) \to \FEngel(M) \]
of the space of Engel structures into the space of formal Engel structures is surjective in all homotopy groups.
\end{proposition}
The proposition completely solves the existence problem for Engel structures. This note intends to shed some light on their classification.

\subsection{Curves in the $2$--sphere and the Engel local model} \label{ssec:key}

Consider a (possibly open) $3$--manifold $N$. We focus our attention now on Engel structures on the product manifold $N \times [0,1]$ with coordinates $(p,t)$. We are interested in those $2$--distributions $\SD$ of the form $\langle \partial_t \rangle \oplus X$, with $X$ a vector field tangent to the slices $N \times \{t\}$. We write $X' = [\partial_t,X]$ and $X'' = [\partial,X']$.

Given some point $p \in N$, one can use the flow of $\partial_t$ to define a trivialisation that identifies $\NS(T_{(p,t)}N \times \{t\})$ with $\NS^2$ in a $t$--independent fashion. Along each curve $\{p\} \times [0,1]$, the vector fields $X,X',X''$ can be regarded as curves $X_p,X_p',X_p'': [0,1] \to \NS^2$. Given a curve in $\NS^2$, we say that one of its points is an \textsl{inflection point} if the normal curvature of the curve is vanishing at the point. If the curvature is everywhere positive, the curve is said to be \emph{convex}; if it is everywhere negative, it is said to be \emph{concave}.

\begin{proposition}[\cite{CPPP}] \label{prop:key}
The $2$--distribution $\SD = \langle \partial_t \rangle \oplus X$ in $N \times [0,1]$ is not integrable at a point $(p,t)$ if and only if $X_p$ is immersed at time $t$.

$\SD$ is Engel at $(p,t)$ if and only if, additionally, at least one of the following two conditions holds:
\begin{enumerate}
  \item the curve $X_p$ has no inflection point at time $t$,
  \item $\langle X_q(t),X_q'(t) \rangle = [\SD, \SD] \cap T(N \times \{t\})$ is a contact structure in $\Op(p)\times\{t\}$.
\end{enumerate}
\end{proposition}
The first condition does not depend on the framing chosen for $TN$, since convexity is preserved by transformations in $\PGL(2)$. We will normally choose framings that make the second condition easy to check.

\begin{remark}
We can also understand formal Engel structures in this setting. Indeed, suppose we have a oriented flag $\SW \subset \SD \subset \SE$ with $\SW = \langle \partial_t \rangle$ in $\D^3 \times [0,1]$. Then, it can be regarded as a $\D^3$--family of \emph{formal immersions} of $[0,1]$ into $\NS^2$. Proposition \ref{prop:key} shows that they are immersions if and only if $\SE = [\SD,\SD]$. If they are additionally convex or concave, the plane field $\SD$ is Engel.
\end{remark}

\subsection{Prolongations} \label{ssec:prolongations}

The canonical examples of Engel structures can be understood within the framework of Proposition \ref{prop:key}:

\begin{example} \label{ex:Cartan}
Let $(N, \xi)$ be a contact $3$--manifold. The total space of the $\mathbb{S}^1$--bundle $\pi: \mathbb{S}(\xi) \to N$ carries an Engel structure given by the universal family construction, called the (oriented) \textsl{Cartan prolongation}. Recall that points in $\mathbb{S}(\xi)$ are pairs $(p,L)$ with $p \in N$ and $L$ an oriented line in $\xi_p$. The Engel structure is simply $\SD_{(p,L)} = d_{(p,L)}\pi^{-1}(L)$. 

In particular, given a disc $\D^3 \subset N$, one can select a framing $\{Y,Z\}$ for $\xi$. Then, in $\D^3 \times \mathbb{S}^1$, with coordinates $(p,L = [\cos(t)Y + \sin(t)Z])$, $t \in [0,2\pi)$, the definition above yields the following structure:
\[ \SD_{(p,L)} = \langle \partial_t, \cos(t)Y + \sin(t)Z \rangle \]
which satisfies the second condition from Proposition \ref{prop:key}. Note that $\SW = \langle \partial_t \rangle$.
\end{example}

\begin{example}
Let $(N, g)$ be a lorentzian manifold of signature $(2,1)$. We denote by $C \subset TN$ the subset given at each point $p \in N$ by the light--like cone $C_p$. To $C$ one can associate the $\NS^1$--bundle $\pi: \NS(C) \to N$ given by quotienting using the $(\R \setminus \{0\})$--action of rescaling. It can be endowed with a canonical Engel structure $\SD(p,L) = d_{(p,L)}\pi^{-1}(L)$, where $L$ is a line in $C_p$. $\SD$ is called the \textsl{Lorentz prolongation}.

Find a disc $\D^3 \subset N$ and choose a orthonormal framing $\{V,Y,Z\}$ with $Y$ and $Z$ space--like, and $V$ time--like. Then, the construction we just described can be written down as:
\[ \SD_{(p,L)} = \langle \partial_t, V + \cos(t)Y + \sin(t)Z \rangle, \quad t \in [0,2\pi). \]
It satisfies the first condition from Proposition \ref{prop:key}. Unlike the previous example, $\SW$ is always transverse to $\langle \partial_t \rangle$.
\end{example}

Up to homotopy, there is a well defined plane associated to each lorentzian metric: any plane that is space--like for the metric. We will be interested in considering lorentzian metrics whose planes are in the same homotopy class as some given contact structure.

\section{The space of Cartan prolongations} \label{sec:prolongations}

For the rest of the article fix a closed, orientable $3$--manifold $N$. Denote by $\Cont$ the space of oriented contact structures on it. It naturally decomposes into several components $\Cont(c)$ corresponding to contact structures having a particular Euler class $c \in H^2(N, \mathbb{Z})$. We can further denote $\Cont(\xi)$ for the connected component containing the contact structure $\xi \in \Cont$. 

Each oriented $\mathbb{S}^1$--bundle over $N$ is given by its Euler class $c$; denote its total space by $N(c)$. Suggestively, denote $\Cartan(c)$ for the space of all oriented Engel structures on $N(c)$ having the fibre direction as their kernel. Write $\pi: N(c) \to N$ for the projection. Any Engel structure $\SD \in \Cartan(c)$ defines a contact structure $\xi = d\pi(\SE)$ on $N$, since $\SW = \ker(d\pi)$. Orient the line field $\SW$ using the orientation of the fibre. Then, $\xi$ inherits an orientation from $\SE/\SW$. Hence, there is a projection:
\[ \Cartan(c) \to \Cont. \] 

It is immediate that the Euler class of $\xi$ must be of the form $kc$, with $k>0$. This integer is called the \textsl{turning number} and is computed as follows. Take any $\NS^1$--fibre of $N(c)$. Find some $\NS^1$--invariant, positively oriented framing of $\SE/\SW$. Compute the degree of $\SD/\SW$ with respect to this framing. The resulting number $k$ does not depend on the choices involved and is necessarily positive.

Denote by $\Cartan(c,k) \subset \Cartan(c)$ the space of Cartan prolongations having turning number $k$. Write $\Cartan(c,k,\xi) \subset \Cartan(c,k)$ for the subspace of those that additionally project down to $\xi \in \Cont$. Observe that a path of Cartan prolongations projects down to a path of contact structures; write $\Cartan(c,k,[\xi])$ for the subspace of those prolongations that lift contact structures homotopic to $\xi$.

Denote by $\Cover(c,k)$ the space of $k$--fold covers from $N(c)$ to $N(kc)$; i.e. positively oriented fibrewise submersions with $k$ sheets lifting the identity on $N$. Once we fix a bundle isomorphism between the sphere bundle of $\xi$ and $N(kc)$, we can construct the following homeomorphism:
\[ f: \Cartan(c,k,\xi) \to \Cover(c,k) \]
\[ f(\SD)(p,L) = (p,[d\pi_p(\SD(p,L))]), \]
where $[d\pi_p(\SD(p,L))]$ denotes the oriented line in $\xi_p$ determined by projecting down $\SD(p,L)$. Note that $f(\SD)$ pulls back the canonical Cartan prolongation in $N(kc) \cong \NS(\xi)$ to $\SD$.

All the contact structures in a neighbourhood of $\xi$ can be identified with $\xi$ itself using a projection along a complementary line field. This implies that the corresponding sphere bundles can consistently be identified with $N(kc)$. This readily implies that 
\[ \Cartan(c,k) \to \Cont(kc) \quad \text{ and } \quad  \Cartan(c,k,[\xi]) \to \Cont(\xi) \]
are locally trivial fibrations with fibre $\Cover(c,k)$. 

Our aim in this section is to understand the homotopy type of the spaces $\Cartan(c,k,[\xi])$ using the fibration structure we have just presented. Before we provide a precise statement, we need some additional setup.

\subsection{Prolongations over a fixed contact structure} \label{ssec:Cartan}

First we will describe the homotopy groups of the space $\Cartan(c,k,\xi) \cong \Cover(c,k)$, the fibre. The $\pi_0$ case was already described in \cite{KS}.

Observe that, by fixing some element $\tau \in \Cover(c,k)$, one readily obtains an inclusion $\Cover(kc,1) \subset \Cover(c,k)$ by making $\tau$ act by pullback. Since $\Cover(kc,1)$ is a group that contains the gauge transformations $\SG(kc)$ of $N(kc)$ as an abelian subgroup, we can regard $\SG(kc)$ as a subspace of $\Cover(c,k)$ as well. 

\begin{lemma} \label{lem:retractions}
For any $\tau \in \Cover(c,k)$, the inclusions
\[ \SG(kc) \to \Cover(kc,1) \to \Cover(c,k) \]
are weak homotopy equivalences.
\end{lemma}
\begin{proof}
Let $\phi: (\D^j,\partial \D^j, 1) \to (\Cover(c,k),\SG(kc), \tau)$ be a continuous function representing an element in $\pi_j(\Cover(c,k),\SG(kc),\tau)$. It is sufficient to show that it retracts to $\SG(kc)$.

Restricted to an $i$--simplex $\Delta_i$ of $N$, the bundles $N(c)$ and $N(kc)$ are trivial. There, $\phi$ can be thought as a $(\Delta_i \times \D^j)$ --family of positively oriented submersions of $\NS^1$ onto itself with $k$--sheets. The $\SO(2)$--bundle structure on $N(c)$ can be taken to be the one induced from $N(kc)$ by using $\tau$, and hence $\tau$ can be assumed to be the map $z^k$ on each fibre; the elements of $\SG(kc) \subset \Cover(c,k)$ are those of the form $\phi \circ z^k$ with $\phi$ a $(\Delta_i \times \D^j)$--family of rotations.

Assume that a suitable homotopy has already been found in the $(i-1)$th skeleton of $N$. The $(\Delta_i \times \D^j)$--family of submersions of $\NS^1$ onto itself can be lifted to define a family in $\Diff^+(\NS^1)$ such that its boundary lies in $\NS^1$, the rotations. Recalling that $\NS^1 \to \Diff^+(\NS^1)$ is a weak homotopy equivalence concludes the inductive step.
\end{proof}

\begin{lemma} \label{lem:homotopy}
The homotopy groups of $\SG(kc)$, and hence of $\Cover(c,k)$, are given by:
\[ \pi_0 = H^1(N, \mathbb{Z}), \qquad \pi_1 = \mathbb{Z}, \qquad \pi_j = 0, \text{ for } j>1. \]
\end{lemma}
\begin{proof}
Recall that $\NS^1$ is the classifying space for the discrete group $\Z$. Then:
\[ \pi_0(\SG(N)) = \pi_0(\Maps(N, \NS^1)) = H^1(N, \mathbb{Z}). \]
In general, it is a result of Thom \cite{Th} that $\pi_j(\Maps(N, K(G,n))) = H^{n-j}(N, G)$. 
\end{proof}

\begin{remark} \label{rem:horizontalDistance}
Lemma \ref{lem:homotopy} can be proved using obstruction theory as in Lemma \ref{lem:retractions}. This is useful to provide a geometrical interpretation of the result. Let us outline the argument, which is similar to the one presented in \cite{KS}. We need to fix a basepoint $\tau \in \Cover(c,k)$.

An explicit identification between $\pi_0(\Cover(c,k))$ and $H^1(N, \mathbb{Z})$ can be given as follows. Take an element $\nu \in \Cover(c,k)$. Over each loop $\gamma \subset N$, the bundles $N(c)$ and $N(kc)$ trivialise. Given any section $s \in \Gamma(N(c)|_\gamma)$, one can compute the degree of $\nu(s)$ with respect to $\tau(s)$. This gives a homomorphism $H_1(N,\Z) \to \Z$ and thus an element in $H^1(N, \mathbb{Z})$. This element only depends on the connected component of $\nu$; we call it the \textsl{horizontal distance} between $\tau$ and $\nu$.

Similarly, let $\nu_t \in \Cover(c,k)$, $t \in \NS^1$, be a loop with $\nu_1 = \tau$. Take a point $p \in N$ and lift it to a point $P \in N(c)$. We say that the degree of $t \to \nu_t(P)$, $t \in \NS^1$, as a loop in the fibre of $N(kc)$ over $p$, is the \emph{looping number}. This identifies $\pi_1(\Cover(c,k))$ with $\Z$.
\end{remark}

Let $\xi$ be a contact structure with Euler class $kc$. Recall that our objective is to understand the homotopy type of $\Cartan(c,k,[\xi])$. From these lemmas and the homotopy long exact sequence for the fibration, it easily follows that:
\[ \pi_j(\Cartan(c,k,[\xi])) = \pi_j(\Cont(\xi)), \quad \text{ for } j>2. \]
However, the cases of $\pi_0$, $\pi_1$, and $\pi_2$ are more subtle. The key is understanding the connecting morphism
\[ \pi_j(\Cont(\xi)) \to \pi_{j-1}(\Cover(c,k)) \qquad j=1,2, \]
which is not zero in general.

\subsection{Formal Cartan prolongations} \label{ssec:FCartan}

We will now introduce two spaces of geometrical structures whose homotopy groups are easy to compute. We will be able to regard $\Cartan(c,k,[\xi])$ as a subspace within them. This will allow us to state and prove our main theorem about the spaces $\Cartan(c,k,[\xi])$.

\subsubsection{Prolongations of plane fields}

Denote by $\Planes$ the space of oriented plane fields in $N$. Write $\Planes(c)$ for those of Euler class $c \in H^2(N, \Z)$ and $\Planes(\xi)$ for the connected component containing the plane field $\xi \in \Planes$. By fixing a parallelisation of $N$, $\Planes$ can be readily identified with $\Maps(N, \mathbb{S}^2)$.

We write $\FCartan(c)$ for the space of oriented $2$--distributions in $N(c)$ that contain the fibre direction, are everywhere non--integrable (but not necessarily maximally), and whose induced $3$--distribution obtained by Lie bracket is preserved by flows along the fibre. The elements in $\FCartan(c)$ are called \textsl{formal Cartan prolongations}.

Let $\SD \in \FCartan(c)$. Then, by definition, $\xi = d\pi(\SE = [\SD,\SD])$ is a plane field in $N$; $\xi$ being contact amounts to $\SD$ being an element in $\Cartan(c)$. The orientation of $\SE/\SW$ orients $\xi$, just like in the case of Cartan prolongations. The turning number $k$ can also be defined; we write $\FCartan(c,k) \subset \FCartan(c)$ for the subspace of those formal Cartan prolongations with turning number $k$: they necessarily project down to plane fields of Euler class $kc$. Similarly, write $\FCartan(c,k,[\xi])$ for those lifting plane fields homotopic to $\xi$. 

There are locally trivial fibrations:
\[ \Cover(c,k) \longrightarrow \FCartan(c,k) \longrightarrow \Planes(kc), \]
\[ \Cover(c,k) \longrightarrow \FCartan(c,k,[\xi]) \longrightarrow \Planes(\xi), \]
where $\Cover(c,k)$ is defined as before.

\subsubsection{Prolongations of rank $2$ bundles}

There is a natural inclusion of the space of oriented plane fields into the space of oriented rank $2$ bundles:
\[ \Maps(N,\NS^2) \cong \Planes \to \Bundles \cong \Maps(N,\Gr(2,\infty)), \]
where $\Gr(2,\infty)$ is the infinite Grassmanian of oriented $2$--planes, which is the Eilenberg--Maclane space $K(2,\mathbb{Z})$. We write $\Bundles(c)$ for the subspace of bundles having Euler class $c \in H^2(N,\Z)$. Let $\V(2,\infty)$ be the Stiefel manifold of ordered pairs of orthonormal vectors in $\R^\infty$; recall that there is a tautological fibration 
\[ \NS^1 \to \V(2,\infty) \to \Gr(2,\infty). \]

We will now explain what a prolongation is in this setting. We define $\FCartan^\infty(c)$ to be the space of maps of $N(c)$ into $\V(2,\infty)$ which are lifts of maps $N \to \Gr(2,\infty)$ and are fibrewise submersions respecting the orientation. This space has several components $\FCartan^\infty(c,k)$ distinguished by the Euler class $kc$ of the underlying $2$--plane bundle, $k>0$. 

Let $\xi$ be a oriented plane field of Euler class $kc$. The following diagram commutes:
\vspace{-0.2cm}
\begin{figure}[H] \label{diag:fibrations}
\centering
\begin{tikzpicture}
  \matrix (m) [matrix of math nodes,row sep=1.5em,column sep=4em,minimum width=2em] {
     \Cover(c,k) & \Cartan(c,k,[\xi]) & \Cont(\xi) \\
     \Cover(c,k) & \FCartan(c,k,[\xi]) & \Planes(\xi) \\
     \Cover(c,k) & \FCartan^\infty(c,k) & \Bundles(kc) \\};
  \path[-stealth]
    (m-1-1) edge node [left] {$\cong$} (m-2-1)
            edge node [above] {$$} (m-1-2)
		(m-1-2) edge node [left] {} (m-2-2)
            edge node [above] {$$} (m-1-3)
		(m-1-3) edge node [left] {$$} (m-2-3)
		
		(m-2-1) edge node [left] {$\cong$} (m-3-1)
            edge node [above] {$$} (m-2-2)
		(m-2-2) edge node [left] {} (m-3-2)
            edge node [above] {$$} (m-2-3)
		(m-2-3) edge node [left] {$$} (m-3-3)
		
		(m-3-1) edge node [below] {$$} (m-3-2)
		(m-3-2) edge node [below] {$$} (m-3-3);
\end{tikzpicture}
\end{figure}
\vspace{-0.6cm}
where each row is a fibration.

\subsection{Statement of the theorem}

Let $\xi$ be a contact structure of Euler class $kc$. According to the commutative diagram above, the connecting morphism
\[ \pi_j(\Cont(\xi)) \to \pi_{j-1}(\Cover(c,k)) \]
factors through $\pi_j(\Bundles(kc))$. This motivates us to consider the subgroup:
\[ \pi_j^\trivial(\Cont(\xi)) = \ker(\pi_j(\Cont(\xi)) \to \pi_j(\Bundles(kc))). \]

We can now state the main result of the section.
\begin{theorem} \label{thm:ot}
Let $\xi$ be an overtwisted contact structure of Euler class $kc$. Then:
\begin{align*}
\pi_0(\Cartan(c,k,[\xi])) &= \pi_0(\Cont(\xi)) \times H^1(N,\Z_2) & \\
\pi_1(\Cartan(c,k,[\xi])) &= \pi_1^\trivial(\Cont(\xi)) \times \Z_2 & \\
\pi_2(\Cartan(c,k,[\xi])) &= \pi_2^\trivial(\Cont(\xi)) & \\
\pi_j(\Cartan(c,k,[\xi])) &= \pi_j(\Cont(\xi)) & \text{ if } j>2.
\end{align*}
The term $H^1(N,\Z_2)$ is the mod $2$ reduction of the horizontal distance. Similarly, the term $\Z_2$ is the parity of the looping number.
\end{theorem}

The proof relies on understanding the inclusion $\pi_j(\Cont(\xi)) \to \pi_j(\Bundles(kc))$. For an arbitrary contact structure $\xi$ this is very difficult. However, if $\xi$ is assumed to be overtwisted the problem simplifies considerably. This flexibility is provided by the well--known result of Eliashberg regarding the classification of overtwisted contact structures:
\begin{lemma}[\cite{El}] \label{lem:Eliashberg}
Let $\xi$ be an overtwisted contact structure. The inclusion
\[ \Cont(\xi) \to \Planes(\xi) \]
is surjective in all homotopy groups, where $\xi$ is assumed to be the basepoint. Additionally, this map is a bijection at the level of connected components. 
\end{lemma}
The inclusion is a weak homotopy equivalence if one additionally fixes the overtwisted disc, but we will not need this fact.

\subsection{Proof of the theorem}

\subsubsection{The connecting morphism for bundles}

The next two lemmas show that the connecting morphism is a bijection in the case of bundles. 

\begin{lemma} \label{lem:injectivity}
The connecting morphism $\pi_j(\Bundles(kc)) \to  \pi_{j-1}(\Cover(c,k))$ is injective.
\end{lemma}
\begin{proof}
Equivalently, we will show that the morphism $\pi_j(\FCartan^\infty(c,k)) \to \pi_j(\Bundles(kc))$ is zero. Take any element in $\pi_j(\FCartan^\infty(c,k))$, and find a representative $K \subset \FCartan^\infty(c,k)$. The class $[K]$ maps to $[\xi]$, where $\xi$ is the $j$--sphere of $2$--plane bundles underlying $K$. From the definition of formal Cartan prolongation (for bundles), we have that the circle bundle of $\xi$ is $k$--covered by the trivial family of circle bundles based on $N(c)$. We deduce (for instance, using functoriality of the Euler class) that $\xi$ must be trivial as a family as well, so $[\xi] = 0$.
\end{proof}

Similarly:

\begin{lemma} \label{lem:surjectivity}
The connecting morphism $\pi_j(\Bundles(kc)) \to \pi_{j-1}(\Cover(c,k))$ is surjective.
\end{lemma}
\begin{proof}
Take $\xi$ to be the basepoint in $\Bundles(kc)$ and fix a lift $\tau \in \FCartan^\infty(c,k)$. Take a class $G \in \pi_{j-1}(\Cover(c,k))$, which we can think of as a homotopy class in the gauge transformations of $\xi$, by Lemma \ref{lem:retractions}.

Lemma \ref{lem:homotopy} implies that $G$ is given by a cohomology class $g \in H^{2-j}(N,\Z)$. Similarly, $\xi$ is given, up to bundle isomorphism, by its Euler class $e \in H^2(N,\Z)$. Recall that the K\"unneth formula yields an isomorphism
\[ H^{2-j}(N,\Z) \oplus H^2(N,\Z) \overset{(\alpha,\beta)}\longrightarrow H^2(N \times \NS^j, \Z). \]
Consider the unique, up to homotopy, $j$--sphere $K$ of bundles based on $\xi$ and having $\alpha(g) + \beta(e)$ as its Euler class when regarded as a plane bundle over $N \times \NS^j$. We claim that the connecting morphism maps $[K]$ to $G$.

Take $j=1$. Write $P: N \times \NS^1 \to N$. Write $Q: N \times [0,1] \to N \times \NS^1$ for the obvious quotient map. There is a unique, up to homotopy, isomorphism between $Q^*K$ and $Q^*P^*\xi$ extending the identification $(Q^*K)|_{N \times \{0\}} = \xi$. The identification of $(Q^*K)|_{N \times \{0\}}$ with $(Q^*K)|_{N \times \{1\}}$ yields a gauge transformation $\phi$ of $\xi$.

We claim that $\phi$ is a representative of $G$. Recall that the Euler class of a $2$--plane bundle over the torus can be computed as follows: find a section over the complement of the meridian $\gamma$ and compare the degrees of the two resulting sections over $\gamma$. Let now $\gamma$ be some embedded loop in $N$, and let $T$ be the corresponding torus on $N \times \NS^1$. By construction, $K|_T$ has Euler class $\alpha(g)|_T$, which implies that $\phi|_\gamma$ is described by $g|_\gamma$. Since gauge transformations are characterised by their action over loops, the claim follows.

The case $j=2$ is similar. In that case, we have to study what happens over a single point $x \in N$ and the corresponding sphere $\{x\} \times \NS^2$. 
\end{proof}

\subsubsection{Non--trivial families of plane fields} \label{ssec:nonTrivialPlanes}

The following proposition shows that there are many families of plane fields which are non--trivial as families of vector bundles.

\begin{proposition} \label{prop:planeEulerClass}
Let $d_j =2v_j \in H^2(N \times \NS^j, \Z)$, $j=1,2$. Fix $\xi \in \Planes$ with Euler class $d_j|_N$. Then, there is a sphere $K_j$ in $\Planes$ based at $\xi$ whose Euler class as a $2$--plane bundle over $N \times \NS^j$ is $d_j$.
\end{proposition}
\begin{proof}
Assume $j=1$. Take a CW--decomposition of $N$ with only one top cell. Take the CW--decomposition of $\NS^1$ with a single $1$--cell and $x$ the unique $0$--simplex. Denote by $\ST$ the product CW--decomposition in $N \times \NS^1$. Write $\ST^* \subset \ST$ for the collection of cells not contained in $N \times \{x\}$. Deform $\xi$ to be constant (as a map into the Grassmannian) over the $1$--skeleton of $N$. We define $(K_1)|_{N \times \{x\}} = \xi$ and we aim to extend it to $\ST^*$. 

Over the $1$--cells, $K_1$ can be defined to be constant, like $\xi$. Over a $2$--cell $\Delta_2$, we define it to be a map into $\NS^2$ of degree $\phi_1(\Delta_2)$, where $[\phi_1] = v_1$. This provides the desired Euler class. Over the $3$--skeleton of $\ST^*$, the obstruction for extending is given by $d\phi_1$ by construction, which evaluates zero over the cells of $\ST^*$. This leaves the single $4$--cell $\Delta_4$ to fill. The obstruction is the degree of the map $K_1: \partial \Delta_4 \to \NS^2$. We can trace back our steps and modify $K_1$ over some $3$--cell to make sure this degree is zero.

Assume $j=2$. Then, the isomorphism:
\[ H^2(\NS^2,\Z) \oplus H^2(N,\Z) \overset{(\alpha,\beta)}\longrightarrow H^2(N \times \NS^2, \Z), \]
indicates that we can simply compute the Euler class of any plane bundle by evaluating separately on $N \times \{x\}$ or $\{p\} \times \NS^2$. Take the manifold $N \times \D^2$: over it, we have (the pullback of) the bundle $TN$ which is trivisalised as the trivial $\R^3$--bundle; as a 2--distribution inside, we define $K_2$, which is $\D^2$--invariant and equal to $\xi$ on every $N \times \{x\}$. We aim two glue two copies of $N \times \D^2$ so that the glued copies of $K_2$ have the desired Euler class when restricted to each $\{p\} \times \NS^2$.

Consider the loop $\NS^1 \to \SO(2)$ realising the Euler class $\alpha^{-1}(d_2) \in H^2(\NS^2,\Z)$ through the clutching construction and denote by $\phi: \NS^1 \to \SO(3)$ its inclusion into $\SO(3)$. Observe that, since $d_2=2v_2$, $\phi$ is contractible in $\SO(3)$. We can define then another map $\Phi: N \times \NS^1 \to \SO(3)$ so that:
\begin{itemize}
\item $\Phi|_{\{p\} \times \NS^1} = \phi$, up to a $\SO(3)$--transformation that only depends on $p$,
\item $\Phi$ fixes (not pointwise) the plane $\xi$.
\end{itemize}
What we are essentially saying is that $\phi$ was a family of rotations of the $XY$--plane that was lifted to $\R^3$, and $\Phi$ is a $p$--dependent family of rotations that looks the same but, instead, the plane that $\Phi|_{\{p\} \times \NS^1}$ rotates is $\xi_p$. Since $\phi$ was contractible, so is $\Phi$, so the resulting $\R^3$--bundle is trivial. However, in each $\{p\} \times \NS^2$ the restriction of $\xi$ has been twisted to have Euler class $\alpha^{-1}(d_2)$, proving the claim. 
\end{proof}

\subsubsection{The proof of the theorem} \label{ssec:nonTrivialContact}

Let $\xi$ be an overtwisted contact structure and fix $\tau \in \Cartan(c,k,\xi)$ a basepoint in the fibre over $\xi$. The existence of $\tau$ identifies the connected components of $\Cartan(c,k,\xi) \cong\Cover(c,k)$ with $H^1(N,\Z)$, as in Lemma \ref{lem:homotopy}.

Let us study the connecting morphism 
\[ \pi_1(\Cont(\xi)) \to \pi_0(\Cover(c,k)). \]
Using Lemma \ref{lem:injectivity} we deduce that its kernel is the space $\pi_1^\trivial(\Cont(\xi))$. Let us compute its image. Let $g \in \Cover(c,k)$ and denote by $\nu$ the corresponding prolongation in $\Cartan(c,k,\xi)$. By Lemma \ref{lem:surjectivity} there is a loop of vector bundles, all of them of Euler class $kc$, producing the class of $g$ through the connecting morphism $\pi_1(\Bundles(kc)) \to \pi_0(\Cover(c,k))$. By Proposition \ref{prop:planeEulerClass}, this loop can be realised by a loop of plane fields based on $\xi$ if and only if $[g] \in H^1(N,\Z)$ is even. Then, Lemma \ref{lem:Eliashberg} allows us to turn this into a loop of contact structures $\xi_t$, $t \in \NS^1$ based on $\xi_1 = \xi$. 

The case $\pi_2(\Cont(\xi)) \to \pi_1(\Cover(c,k))$ is analogous. \hfill $\Box$

\section{The classification up to Engel homotopy} \label{sec:Engel}

The previous section dealt with Cartan prolongations. We will now allow homotopies through more general Engel structures (that are, however, still tangent to the fibre direction, so they can be regarded as \emph{generalised prolongations}). 

\subsection{A warm--up exercise} \label{ssec:exercise}

Let us start by working out a particular case which is of interest and that follows easily using the language of Section \ref{sec:prolongations}.

\begin{theorem} \label{thm:exercise}
Let $K$ a topological space. Let $\phi_0, \phi_1: K \to \Cartan(c)$ be two continuous maps and let $\phi_s: K \to \FCartan(c)$, $s \in [0,1]$, be a homotopy between them. Then $\phi_s$ can be $C^\infty$--approximated, relative to its ends, by a homotopy with image in $\Engel(N(c))$.
\end{theorem}

\begin{figure}[ht] 
\centering
\includegraphics{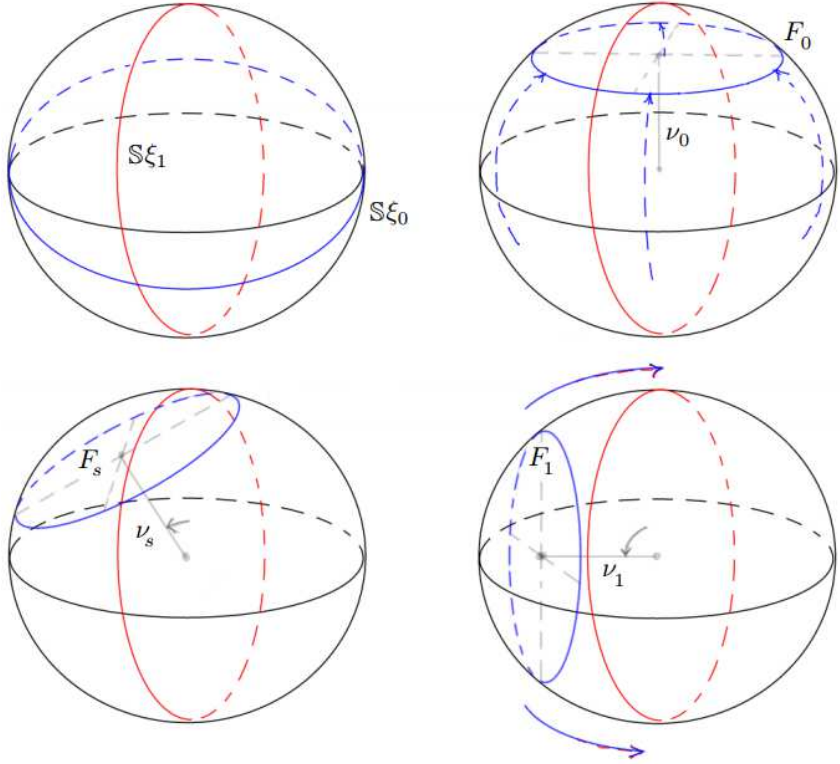}
\caption{$\NS\xi_i$, $i=0,1$, is the circle bundle of the contact structure $\xi_i$. Using a transverse vector field $\nu_i$, we push $\NS\xi_i$ to the cone $F_i$ of a lorentzian metric. $F_0$ and $F_1$ can be connected by a family $F_s$, $s\in [0,1]$, of cones. This produces a homotopy of the corresponding prolongations.}
\label{fig:lorentzian}
\end{figure}

\begin{proof}
For a proof by picture, refer to Figure \ref{fig:lorentzian}. Consider the maps $\phi_s$. There is a corresponding family $K \times [0,1]$ of maps 
\[ f_{x,s}: N(c) \to \mathbb{S}TN \]
\[ f_{x,s}(p,L) = d_{p,L}\pi([\phi_s(x)]) \]
which are simply the tautological maps associated to each formal Cartan prolongation. Write $\xi_{x,s}$ for the oriented contact plane associated to $\phi_s(x)$.

Write $\nu_{x,s}$ for a family of unit vectors in $N$ such that $\nu_{x,s}$ is orthogonal to $\xi_{x,s}$ for each $(x,s) \in K \times [0,1]$. Define a function $h: [0,1] \to \R$ vanishing to all orders on $0$ and $1$ and otherwise satisfying $h(s) > 0$, $s \in (0,1)$. Consider the following deformation of $f$:
\[ F_{x,s} = \dfrac{f_{x,s} + h(s)\nu_{x,s}}{|f_{x,s} + h(s)\nu_{x,s}|}. \]
The tautological distributions associated to $F_{x,s}$ provide a family $\psi_s:K \to \Engel(N(c))$, $s \in [0,1]$; the Engel structures $\psi_s$ are Lorentz prolongations if and only if $s \in (0,1)$. Making $h(s)$ approach zero, $\psi_s$ becomes arbitrarily close to $\phi_s$.
\end{proof}

In particular, the theorem indicates that Engel structures do not seem to recall global contact topology information. This is consistent with the fact that there are Engel cobordisms between contact structures homotopic only as plane fields (see \cite{CPPP}).

\begin{remark}
Theorem \ref{thm:exercise} shows that we can think of Lorentz prolongations as convex push--offs of formal Cartan prolongations. In particular, each Lorentz prolongation has a well defined turning number.
\end{remark}

\subsection{Statement of the main theorem} \label{ssec:main}

Our main result refers to both Cartan and Lorentz prolongations. It reads:
\begin{theorem} \label{thm:main}
Let $K$ be a CW--complex. Let $\phi_0,\phi_1: K \to \Engel(N(c))$ be two continuous maps with image either in the oriented Cartan prolongations or in the Lorentz prolongations. Suppose that both of them have turning number greater or equal to $6$. Then, they are Engel homotopic if and only if they are formally homotopic.
\end{theorem}

If $c=0$, the bound on the turning numbers can be improved and the proof is actually simpler:
\begin{proposition} \label{prop:main}
Let $K$ be a CW--complex. Let $\phi_0,\phi_1: K \to \Engel(N(0))$ be two continuous maps with image either in the oriented Cartan prolongations or in the Lorentz prolongations. Suppose that both of them have turning number greater or equal to $2$. Then, they are Engel homotopic if and only if they are formally homotopic.
\end{proposition}

The main ingredient in the proof is the interplay between Engel structures and families of curves in $\NS^2$, as discussed in Subsection \ref{ssec:key}. We will introduce the technical results we need first.

\subsection{Curves in $\NS^2$} \label{ssec:curves}

A curve $\gamma: \NS^1 \to \NS^2$ having no inflection points has an associated Frenet map $\Gamma_\gamma: \NS^1 \to \text{O}(3)$ given at $p$ by the matrix $(\gamma(p), \overset{.}{\gamma}(p)/|\overset{.}{\gamma}(p)|, \mathfrak{n}(p))$, with $\mathfrak{n}: \NS^1 \to \NS^2$ satisfying $\langle\overset{..}{\gamma}(p), \mathfrak{n}(p)\rangle > 0$. We say that $\gamma$ is convex if this matrix lives in $\SO(3)$. We can still define the Frenet map of an immersed curve by requiring $\mathfrak{n}(p)$ to be the unique vector making it lie in $\SO(3)$. Let $\Imm$ be the space of immersions of $\NS^1$ into $\NS^2$. Its formal counterpart $\FImm$, the space of formal immersions, can be identified with $\Maps(\NS^1, \SO(3))$. Denote by $\SL \subset \Imm$ the subspace of convex curves. 

\subsubsection{A result of Little/Saldanha} \label{sssec:Little}

Let $L$ be some compact CW--complex, and let $n$ be a positive integer. Fix maps $f: L \to \Imm$, $t: L \to \NS^1$. We construct a new map $f^{[t\# n]}: L \to \Imm$ using the following procedure: for each $p \in L$ the curve $f(p)$ can be cut at the point $f(p)(t(p))$ and be modified by adding $n$ small convex loops. This can be done continuously on $p$. It can also be done over different points as long as we have functions $t_0,\dots,t_m: L \to \NS^1$ with disjoint image; we then write $f^{[t_0\#n_0,\dots,t_m\#n_m]}$ for the resulting family. Note that this family is certainly not unique; by taking the loops small, it can be assumed to be $C^0$--close to $f$.

The following lemma summarises the facts we need about this operation:
\begin{lemma}[Little,Saldanha] \label{lem:Little}
Let $L$ be a compact CW--complex. Let $f: L \to \Imm$. Let $t,t_0,\dots,t_m: L \to \NS^1$ be homotopic functions with disjoint image. Let $n,n_0,\dots,n_m$ be positive integers. The following statements hold:
\begin{itemize}
\item[a.] If $f$ is convex, so is $f^{[t_0\#n_0,\dots,t_m\#n_m]}$.
\item[b.] Sliding the cutting points provides a homotopy between $f^{[t\#n_0+\dots n_m]}$ and $f^{[t_0\#n_0,\dots,t_m\#n_m]}$ through immersions. If $f$ is convex, they are homotopic as convex curves.
\item[c.] There is a homotopy between $f$ and $f^{[t\#2]}$ as immersions. This homotopy takes place in a small neighbourhood of the point $t$.
\item[d.] Assume that $f$ is convex. Then $f^{[t\#1]}$ and $f^{[t\#3]}$ are homotopic as convex curves. This homotopy takes place in a small neighbourhood of the point $t$.
\item[e.] The homotopies between $f^{[t\#1]}$ and $f^{[t\#3]}$ produced by Statements (c.) and (d.) are homotopic to one another through immersions, relative to the ends.
\item[f.] If $f$ is fixed and the collection $\{t_0,\dots,t_m\}$ is sufficiently dense in $\NS^1$, the curve $f^{[t_0\#1,\dots,t_m\#1]}$ can be chosen to be convex.
\end{itemize}
\end{lemma}
These statements can be found in \cite{Sal}[Section 6], but the techniques involved appeared already in \cite{Li}. In Figure \ref{fig:Little} a explicit homotopy between a convex curve having winding 2 in an affine chart and another one having winding 4, is shown. This homotopy can be adapted to prove Lemma \ref{lem:Little}.

\begin{figure}[ht] 
\centering
\includegraphics{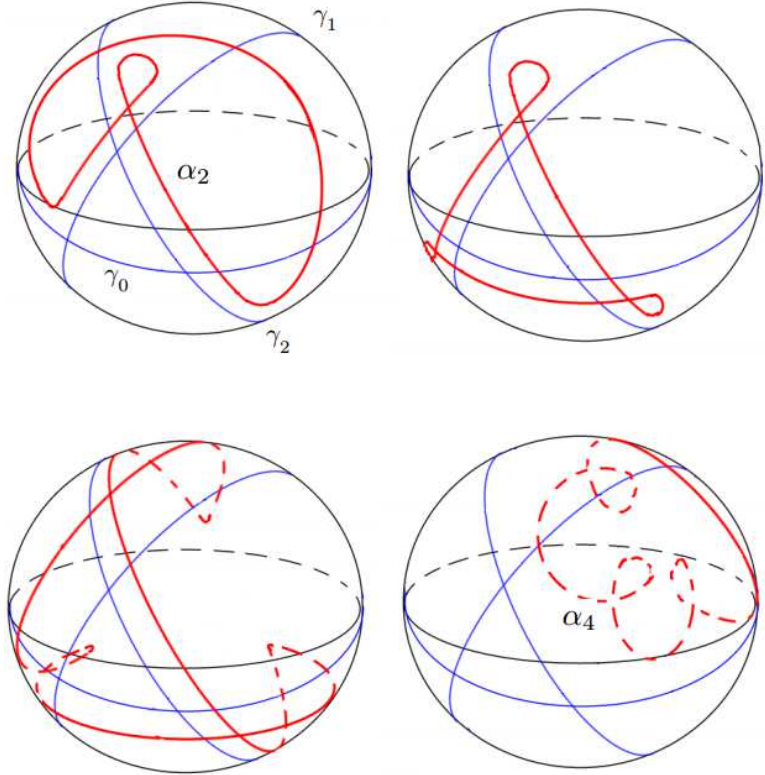}
\caption{The curves $\gamma_i$, $i=0,1,2$, are maximal circles. The curve $\alpha_2$ is convex and, in the frontal hemisphere, has winding number $2$. By pushing the upper strand down, it can be taken to the second figure. It is comprised of three segments that are convex pushoffs of the $\gamma_i$ whose corners have been rounded to preserve convexity. The third figure is obtained from the second by following the $\gamma_i$ for a longer time. Pushing everything to the opposite hemisphere yields a curve $\alpha_4$ with winding $4$.}
\label{fig:Little}
\end{figure}

\subsection{Proof of the main theorem} \label{ssec:proofMain}

The proof can be broken down in several steps. We fix an orientation of $N$.

\subsubsection{Step I. Passing to Lorentz prolongations} \label{sssec:stepI}

Let $i$ be either $0$ or $1$. Assume first that $\phi_i$ is a family of Cartan prolongations. Then, $\phi_i(x)$, $x \in K$, defines an oriented contact plane $\xi_{x,i}$ and a tautological map $f_{x,i}: N(c) \to \mathbb{S}TN$. We can take $\nu_{x,i}$ to be the unique vector field such that $(\xi_{x,i},\nu_{x,i})$ is positively oriented. Then $(f_{x,i} + \nu_{x,i})/|f_{x,i} + \nu_{x,i}|$ defines a family $\psi_i$ of Lorentz prolongations. This family is Engel homotopic to $\phi_i$. Furthermore, the choice of $\nu_{x,i}$ (as opposed to its negative), implies that when we follow the fibres of $N(c)$ positively, the curves that describe $\psi_i(x)$ are convex (i.e. having positive curvature), as opposed to concave.

Assume otherwise that $\phi_i$ is a family of Lorentz prolongations. If the curves describing it are convex, we are done. Otherwise, consider the tautological map $f_{x,i}: N(c) \to \mathbb{S}TN$ associated to $\phi_i$. There are a plane field $\xi_{x,i}$ and a vector field $\nu_{x,i}$ transverse to it so that $\phi_i$ is precisely given by $(f_{x,i} + \nu_{x,i})/|f_{x,i} + \nu_{x,i}|$. We can then find homotopies $\xi_{x,i,s}$ and $\nu_{x,i,s}$, $s \in [0,1]$, so that:
\begin{itemize}
\item $\xi_{x,i,0} = \xi_{x,i}$ and $\nu_{x,i,0} = \nu_{x,i}$, 
\item $\nu_{x,i,s}$ is transverse to $\xi_{x,i,s}$,
\item $\xi_{x,i,1}$ is an overtwisted contact structure.
\end{itemize}
Set $\phi_{i,s}$ to be the Lorentz prolongation obtained by pushing the formal Cartan prolongation of $\xi_{x,i,s}$ with $\nu_{x,i,s}$. This provides a homotopy of $\phi_i$ through Lorentz prolongations. Now, $\phi_{i,1}$ is clearly Engel homotopic to the Cartan prolongation of $\xi_{x,i,1}$ and we can apply the previous discussion. Effectively, we pass through Cartan prolongations to go from \emph{concave} curves to \emph{convex}. Let us henceforth assume that we are dealing with Lorentz prolongations $\psi_i: K \to \Engel(N(c))$, $i \in \{0,1\}$, described by convex curves.

\subsubsection{Step II. Obtaining a non--integrable homotopy}

Fix a parallelisation of $N$, and lift it to a parallelisation of $N(c)$. This provides an almost--quaternionic structure in $TN(c)$. Let $\psi_s: K \to \FEngel(N(c))$, $s \in [0,1]$, be the formal homotopy between $\psi_0$ and $\psi_1$. We can use the almost--quaternionic structure to assume that each $\psi_s$ (as a plane field) contains the fibre direction. By possibly modifying $\psi_s$, the fibre direction can be taken to be \emph{transverse} to the line field of the formal flag (this is the case for $\psi_0$ and $\psi_1$).

Suppose $c=0$. Reasoning as in Subsection \ref{ssec:key} shows that $\psi_s$ can be regarded as a map $N \times K \to \FImm$. By applying the Hirsch--Smale theorem, relative to $s=0,1$, we can assume that $\psi_s$ maps into $\Imm$. If $c$ is not zero, we can still do this over any $3$--ball in $N$. Since the Hirsch--Smale theorem is a full (in particular, relative in the parameter) $h$--principle, we can apply it sequentially using a covering of $N$. Hence, we can assume that the $\psi_s$ are non--integrable, but not necessarily maximally.

\subsubsection{Step III. The vanishing Euler class case}

We will prove Proposition \ref{prop:main} first, since the proof is simple but showcases how all the ingredients are used.

Consider the family $\psi_s$, $s \in [0,1]$. The arguments above show that in the $c=0$ case we can simply regard it as a family $\Psi: K \times N \times [0,1] \to \Imm$. Using Lemma \ref{lem:Little} (f.) we can find an even integer $m$ and a collection of points $t_0,\dots,t_m \in \NS^1$ such that the family $\Psi^{[t_0\#1,\dots,t_m\#1]}$ has image in $\SL$ after a small homotopy. However, the families $\Psi^{[t_0\#1,\dots,t_m\#1]}|_{K \times N \times \{0,1\}}$ and $\Psi|_{K \times N \times \{0,1\}}$ do not agree. 

If the turning number is at least two, the family $\Psi|_{K \times N \times \{0,1\}}$ is already of the form $\Phi^{[t\#1]}$, where $\Phi: K \times N \times \{0,1\} \to \SL$ can be understood as a family of Lorentz prolongations with turning number $c-1$. An application of Lemma \ref{lem:Little} (b.) and (d.) shows that $\Psi^{[t_0\#1,\dots,t_m\#1]}|_{K \times N \times \{0,1\}}$ and $\Psi|_{K \times N \times \{0,1\}}$ are homotopic as families of convex curves, proving the claim. \hfill $\Box$

The main point is that $c=0$ allows us to assume that $N(c)$ has a non--vanishing section (whose role is played by the point $t$). In the general case we will have to deal with this fact.

\subsubsection{Step IV. Construction of a covering}

Consider the families $\psi_0$ and $\psi_1$. By assumption, they are comprised of Lorentz prolongations with turning numbers $k_0,k_1\geq6$. Regard $N(c)$ as a principal $\NS^1$--bundle. The contractibility of the pair $(\Diff(\NS^1),\NS^1)$ implies that, after a homotopy, we can assume for $\psi_i$ to be invariant under the action of $\Z_{k_i}$ (acting by rotations on the fibre); this follows as in Lemma \ref{lem:retractions}. Let $k = \min(k_0,k_1) \geq 6$.

Given some section $\s: \SU \to N(c)$ over an open set $\SU \subset N$, denote by $I_\s$ the submanifold of $N(c)$ that, on each fibre, is given by moving from $\s$ to $e^{2\pi i/k}\s$ positively. We want to find a covering $\{\SU_j\}_{j=0,\dots,J}$ of $N$ and sections $\s_j: \SU_j \to N(c)$ such that the $I_{\s_j}$ are all disjoint.

Let $\gamma$ be a knot in $N$ representing the Poincar\'e dual of $c \neq 0 \in H^2(N,\Z)$ and let $\nu(\gamma)$ be a tubular neighbourhood. Let $\s_1$ be a section of $N(c)$ over $\nu(\gamma)$. Let $\s_0$ be a (transverse to zero) section of the disc bundle associated to $N(c)$ whose zeroes are $\gamma$; regard it as a section of $N(c)$ away from $\gamma$.

Let $(\alpha;r,\theta)$ be the coordinates in the solid torus $\NS^1 \times \D^2$; fix a diffeomorphism $\nu(\gamma) \cong \NS^1 \times \D^2$. The section $\s_1$ yields an identification of $N(c)|_{\nu(\gamma)}$ with the trivial principal $\NS^1$--bundle over $\NS^1 \times \D^2$. It can be chosen so that $\s_1$ is the constant section $1 \in \NS^1$, and $\s_0(\alpha;r,\theta) = \theta \in \NS^1$ for $r \in \Op(\{1\})$.

Fix $\delta>0$ small. Let $\SU_0$ be the union of the complement of $\nu(\gamma)$ and $\{r>1-2\delta\}$. Let $\SU_1$ be $\NS^1 \times \D^2_{1-3\delta}$. Triangulate $\partial\D^2_{1-5\delta/2}$ and use this to produce a covering $\{U_j\}_{j=2,\dots,J}$ of $\Op(\{1-4\delta < r < 1-\delta\})$ with no triple intersections. Set $\SU_j = \NS^1 \times U_j$. Since the regions $\SU_j$ can be assumed to be arbitrarily thin, the section $\s_0$ is almost constant over each one of them. Due to our assumption on the turning numbers, $I_{\s_0}$ and $I_{\s_1}$ together cover at most a third of any fiber. We deduce that for each $\SU_j$ corresponding to a vertex, we can choose $\s_j$ to be constant and satisfying the claim. Having fixed those, each $\SU_j$ corresponding to an edge intersects $\SU_0$, $\SU_1$, and two of the vertex regions; we deduce that there is some constant $\s_j$ such that $I_{\s_j}$ avoids the corresponding $I_{\s_{j'}}$.

\subsubsection{Step V. Concluding the proof.}

Assume that the family $\psi_s$ is equal to $\psi_0$ in $[0,3\rho]$ and equal to $\psi_1$ in $[1-3\rho,1]$, for $\rho>0$ small. We will modify $\psi_s$ over each $\SU_j$, inductively on $j$. The constructions that follow depend on a large integer $C$; it will be fixed at the end of the proof to ensure that our claims hold.

Write $\psi_s'$ for the family of structures obtained in the step $j-1$. Over $\Op(\SU_j)$, regard it as a family of curves $\Psi_j: \Op(\SU_j) \times K \times [0,1] \to \Imm$. Replace $\Psi_j$ by $\Psi_j^{[\s_j\#2C]}$ in $\SU_j \times K \times [0,1]$ and use the region $(\Op(\SU_j) \setminus \SU_j) \times K \times [0,1]$ to interpolate back to $\Psi_j$. We apply Lemma \ref{lem:Little} (d.) in $[0,\rho] \cup [1-\rho,1]$, (c.) in $[2\rho,3\rho] \cup [1-3\rho, 1-2\rho]$, and (e.) in $[\rho,2\rho] \cup [1-2\rho, 1-\rho]$. We do this for all $j$ and we write $\psi_s'$ for the resulting family. Lemma \ref{lem:Little} (d.) states that $\psi_i'$ and $\psi_i$, $i=0,1$, are Engel homotopic. However, $\psi_i'$ has $2C$ loops added at the points $\s_j$ over $\SU_j$. Note that the curves describing $\psi_i'$ have length bounded above independently of $C$, since the homotopies that add loops in the interpolation region can be done sequentially.

We have to further modify $\psi_s'$, again inductively on $j$. Shrink slightly the $\SU_j$ so that they remain a covering and restrict $I_{\s_j}$ to $\Op(\SU_j)$. Write $\psi_s''$ for the family of structures obtained in the step $j-1$, and let $\Psi_j: \Op(\SU_j) \times K \times [0,1] \to \Imm$ be the corresponding family of curves over $\Op(\SU_j)$. Denote by $I_{\s_j}'$ the subset of $N(c)$ obtained from $I_{\s_j}$ by enlarging it maximally (on each fibre) while keeping it disjoint from $I_{\s_{j'}}$, $j'>j$, and from itself. $I_{\s_j}$ can be enlarged fibrewise, remaining a submanifold, to cover arbitrarily much of $I_{\s_j}'$; redefine it as such. Thanks to the argument in the previous paragraph, $\Psi_j$ is of the form $F^{[\s_j\#2C]}$. Use Lemma \ref{lem:Little} (b.) to replace $\Psi_j$ by $F^{[\s_{j,1}\#1,\dots,\s_{j,2C}\#1]}$ in $\SU_j$, and interpolate back to $\Psi_j$ in $\Op(\SU_j) \setminus \SU_j$. The sections $\s_{j,i}$ are distributed in $I_{\s_j}$ so that they become dense as $C$ goes to infinity.

We write $\psi_s''$ for the resulting family after iterating over all the $\SU_j$. Now, since the $\SU_j$ cover $N$, each curve describing $\psi_s''$ is obtained from an immersed curve by adding loops at a collection of points that becomes dense with $C$; further, the length of this immersed curve is controlled. Lemma \ref{lem:Little} (f.) then implies that, for $C$ large, the curves are convex and hence the homotopy is through Engel structures. This concludes the proof. \hfill$\Box$

\end{document}